\documentclass[3p]{elsarticle}
\usepackage[utf8]{inputenc}
\usepackage{amsmath}
\usepackage{amssymb}
\usepackage{comment}
\usepackage{mathtools}
\usepackage{amsthm}
\usepackage{graphicx}
\usepackage{array}
\usepackage{enumerate}
\usepackage{hyperref}
\usepackage[nameinlink]{cleveref}

\newtheorem{thm}{Theorem}
\newtheorem{lem}[thm]{Lemma}
\newdefinition{rmk}[thm]{Remark}
\crefname{lem}{Lemma}{Lemmas}
\crefname{thm}{Theorem}{Theorems}
\crefname{rmk}{Remark}{Remarks}
\newproof{pf}{Proof}

\usepackage{etoolbox}
\patchcmd{\emailauthor}{(#2)}{}{}{}

\begin{document}

\begin{frontmatter}
\title{On maximal autocorrelations of Rudin-Shapiro sequences}
\author{Daniel Tarnu}
\ead{daniel_tarnu@sfu.ca}
\address{Department of Mathematics, Simon Fraser University, 8888 University Drive, Burnaby, British Columbia V5A
1S6, Canada}
\begin{abstract}
    \indent In this paper, we present an alternative proof showing that the maximal aperiodic autocorrelation of the $m$-th Rudin-Shapiro sequence is of the same order as $\lambda^{m}$, where $\lambda$ is the real root of $x^{3} + x^{2} - 2x - 4$. This result was originally proven by Allouche, Choi, Denise, Erd\'elyi, and Saffari (2019) and Choi (2020) using a translation of the problem into linear algebra. Our approach simplifies this linear algebraic translation and provides another method of dealing with the computations given by Choi. Additionally, we prove an analogous result for the maximal periodic autocorrelation of the $m$-th Rudin-Shapiro sequence. We conclude with a discussion on the connection between the proofs given and joint spectral radius theory, as well as a couple of conjectures on which autocorrelations are maximal.
\end{abstract}
\end{frontmatter}

\section{Introduction}

A length-$n$ binary sequence is defined to be an element of $\lbrace -1, 1 \rbrace^{n}$.  Consider a length-$n$ binary sequence sequence $\textbf{s} = (s_{0},s_{1},\dots,s_{n-1})$. We define the \emph{aperiodic autocorrelation of }\textbf{s}\emph{ at shift $k$} to be
\begin{equation*}
    \sum_{i=0}^{n-1} s_{i}s_{i+k}, 
\end{equation*} 
where $\textbf{s}_{i} = 0$ for $i \notin [0,n-1]$. Likewise, we can define the \emph{periodic autocorrelation of }$\textbf{s}$\emph{ at shift $k$} to be the same sum but taking $\textbf{s}_{i} = \textbf{s}_{i \text{ mod } n}$ for all $i \in \mathbb{Z}$. For instance, if we take $\textbf{s} = (-1,1,1,-1)$, we have
\begin{align*}
    \begin{minipage}{.5\textwidth}
        \centering
        \text{aperiodic autocorrelation of $\textbf{s}$ at shift 2:}
        \begin{tabular}{ccccccc}
            & & & & & & \\
            & -1 & 1 & 1 & -1 &  &  \\
            $\times$ &   &   & -1 & 1 & 1 & -1 \\
            \hline
            & & & -1 & -1 & & 
        \end{tabular}
        \\
        \vspace{0.25cm}
        (-1) + (-1) = -2,
    \end{minipage}
    \begin{minipage}{.5\textwidth}
        \centering
        \text{periodic autocorrelation of $\textbf{s}$ at shift 2:}
        \begin{tabular}{ccccc}
            & & & & \\
            & -1 & 1 & 1 & -1  \\
            $\times$ & 1  & -1  & -1 & 1 \\
            \hline
            & -1 & -1 & -1 & -1 
        \end{tabular}
        \\
        \vspace{0.25cm}
        (-1) + (-1) + (-1) + (-1) = -4,
    \end{minipage}
\end{align*}
where the multiplication is understood to be component-wise. Sometimes, aperiodic autocorrelation and periodic autocorrelation are referred to as \textit{acyclic} autocorrelation and \textit{cyclic} autocorrelation, respectively. Both notions of autocorrelation are used as measures of the similarity of a sequence to its translates. Of course, periodic autocorrelations can be expressed as the sum of two aperiodic autocorrelations. In some contexts (e.g. signal processing), it is desirable to find binary sequences with small autocorrelation at every shift. We will concentrate on studying the maximal aperiodic autocorrelation of Rudin-Shapiro sequences.
\\
\\
\indent We define the \emph{$m$-th Rudin-Shapiro sequence} $(a_{0},a_{1},\dots,a_{2^{m}-1})$ by
\[ a_{i} = (-1)^{\text{\# of pairs of consecutive ones in the binary expansion of $i$}}. \]
For example, we have
\begin{align*}
\text{1st Rudin-Shapiro sequence} &= (1,1), \\
\text{2nd Rudin-Shapiro sequence} &= (1,1,1,-1), \\
\text{3rd Rudin-Shapiro sequence} &= (1,1,1,-1,1,1,-1,1).
\end{align*}
Associated with the $m$-th Rudin-Shapiro sequence is the $m$-th Shapiro polynomial: 
\[ q_{m} := \sum_{j=0}^{2^{m}-1} a_{j}x^{j}, \]
which is an example of a Littlewood polynomial (i.e. a polynomial with coefficients in $\lbrace -1, 1 \rbrace$). Both objects are widely studied. It is worth noting that there are several popular definitions for the Rudin-Shapiro sequences which use recursion \cite{hoholdt, mauduit, allouche2}, some of which are formulated by first building the Shapiro polynomials and defining the Rudin-Shapiro sequences as the sequences of their coefficients \cite{allouche, katz}. We define
\begin{align*}
    C_{m}(k) &= \text{aperiodic autocorrelation of the $m$-th Rudin-Shapiro sequence at shift $k$}, \\
    P_{m}(k) &= \text{periodic autocorrelation of the $m$-th Rudin-Shapiro sequence at shift $k$}.
\end{align*}
Related to the study of autocorrelations of binary sequences is the notion of a merit factor. In \cite{golay}, Golay defines the \emph{merit factor} of a binary sequence $(s_{0},s_{1},\dots,s_{n-1})$ with aperiodic autocorrelations $C(k)$ as
\[ \frac{n^{2}}{2\sum_{k=1}^{n-1} C(k)^{2}}.\]
In \cite{hoholdt2}, it is shown that an equivalent definition is
\[ \frac{n^{2}}{  \left\| \sum_{j=0}^{n-1} s_{j}z^{j} \right\|_{4}^{4} - n^{2}},\]
where the norm used is the $L^{4}$ norm on the unit circle in $\mathbb{C}$, connecting the well-studied areas of autocorrelations of binary sequences and $L^{p}$ norms of Littlewood polynomials. The expected merit factor of a length-$n$ binary sequence is asymptotically 1 as $n \to \infty$ (see \cite{newman, borwein}). It is of interest to find sequences with large merit factor -- for a broad overview of the merit factor problem, see \cite{jedwab}. In \cite{hoholdt}, it is shown that the merit factor of the $m$-th Rudin-Shapiro sequence tends to 3 as $m \to \infty$. In fact, motivated by Littlewood's computation of $\| q_{m} \|_{4}$ given in \cite{littlewood}, it is shown in \cite{borwein} that the same holds for broader class of sequences constructed from so-called \emph{Rudin-Shapiro-like polynomials}.  It follows that
\begin{equation}\label{squarebound}
    \sum_{k=1}^{2^{m}} C_{m}(k)^{2} \sim \frac{4^{m}}{6}.
\end{equation}
So, the $\ell^{2}$ norm of the sequence of all $C_{m}(k)$ for fixed $m$ is established, and we will now focus on the $\ell^{\infty}$ norm. The main objective of this paper is to provide an alternative proof of the following:
\\
\begin{thm}\label{thm1} For all $m \in \mathbb{N}$, there exist $K_{1},K_{2} > 0$ such that
\[ K_{1}\lambda^{m} \hspace{0.1cm} \leq \hspace{0.1cm} \max_{k \neq 0} \hspace{0.1cm} |C_{m}(k)| \hspace{0.1cm} \leq \hspace{0.1cm} K_{2}\lambda^{m}, \]
where $\lambda = 1.659\cdots$ is the real root of $x^{3}+x^{2}-2x-4$.
\end{thm}
For the rest of this paper, for any functions $f : \mathbb{N} \to \mathbb{R}$ and $g: \mathbb{N} \to \mathbb{R}$, we write 
\[ f(m) \ll g(m) \]
if $f(m) = O(g(m))$ for $m \in \mathbb{N}$. In \cite{katz}, Katz and van der Linden prove \cref{thm1} with the best possible $K_{2}$ (being $5/\lambda^{4}$, which is approximately 0.660) using the language of algebraic number theory. \cref{thm1} was originally proven through \cite{allouche, choi}, in which the first step was translating the problem of showing $\max_{k \neq 0} |C_{m}(k)| \ll \lambda^{m}$ into the problem of showing 
\[ \max_{(a_{1},a_{2},\dots,a_{m})} \left\| \prod_{i=1}^{m} T_{a_{i}} \right\|_{2} \ll \lambda^{m} \]
for $(a_{1},a_{2},\dots,a_{m}) \in \lbrace 1,2,3,4 \rbrace^{m}$ and some $T_{a_{i}} \in \mathbb{Z}^{3 \times 3}$. Allouche et al. in \cite{allouche} reduce the number of matrices considered in the product, which led to showing 
\[ \max_{k \neq 0} |C_{m}(k)| \ll (1.00000100000025 \lambda)^{m} \]
which is very close to the desired result. Additionally, the lower bound of \cref{thm1} is proven in \cite{allouche}. Finally, in \cite{choi}, Choi uses these advances to establish the upper bound of \cref{thm1} with $K_{2} < 3.783$. We note that in \cite{choi}, Choi misinterprets their Theorem 1.1 as being for periodic autocorrelation when it actually concerns aperiodic autocorrelation. This result on $C_{m}(k)$ is used to establish results on the oscillation of the modulus of Shapiro polynomials on the unit circle (see \cite{erdelyi}). We use roughly the same ideas as in \cite{allouche, choi}, although constants are left implicit and the crux of our computations given in \cref{lma4} is simpler than those of the computations given in \cite{katz, choi}, so we have a more easily verifiable proof of \cref{thm1} at the expense of precision -- for explicit bounds, see \cite{katz}. We follow this with \cref{thm2}, an analogous result for $P_{m}(k)$, which we roughly state below:
\\
\\
\textbf{Theorem 12.} \textit{For $m \geq 3$, we have that $P_{m}(k)$ is either 0 or $4C_{m-2}(|2^{m-1}-k|)$ depending on $k$, and also that}
\[ \lambda^{m} \ll \max_{k \neq 0} |P_{m}(k)| \ll \lambda^{m}. \]
This is inspired by a private communication by B. Saffari, in which they state that for Rudin-Shapiro sequences, both their periodic autocorrelations and aperiodic autocorrelations behave the same. Afterwards, in Section 3, we discuss how the problem of proving Theorem 1 is the problem of finding the so-called joint spectral radius of $\lbrace MA, MB \rbrace$ which was introduced in \cite{rota}, and provide a heuristic proof of Theorem 1 using an algorithm introduced in \cite{protasov}. We conclude in Section 4 with a couple of conjectures on which $k$ gives maximal $C_{m}(k)$ for fixed $m$. We begin with some preliminary work for the proof of \cref{thm1}.

\section{Proof of the Theorem}

We begin with an overview of the proof of \cref{thm1}. First, we fix $m \in \mathbb{N}$ and $1 \leq k \leq 2^{m}$ and we consider a vector $v_{m} \in \mathbb{R}^{3}$ such that
\[ v_{m} = \begin{bmatrix}
C_{m}(k) \\
\vdots
\end{bmatrix}, \]
where all components depend on $m$ and $k$. Next, we derive a decomposition of $v_{m}$ as a matrix-vector product:
\[ v_{m} = \prod_{i=1}^{m} U^{\alpha_{i}}V^{\beta_{i}} \cdot v \]
for some $U,V \in \mathbb{Z}^{3 \times 3}$ and $(\alpha_{i},\beta_{i}) \in \lbrace (1,0),(0,1) \rbrace$ and $v \in \mathbb{R}^{3}$. The lower bound in \cref{thm1} is proven quickly by diagonalizing $U$. Finally, we prove the upper bound by showing that
\begin{equation}\label{JSR}
\max_{a_{i},b_{i}} \left\| \prod_{i=1}^{m} U^{\alpha_{i}}V^{\beta_{i}} \right\|_{2} \ll \lambda^{m}
\end{equation}
and using the fact that
\[ |C_{m}(k)| \leq \| v_{m} \|_{2} \ll \max_{a_{i},b_{i}} \left\| \prod_{i=1}^{m} U^{\alpha_{i}}V^{\beta_{i}} \right\|_{2}. \]
\\
\\
We now begin developing the actual proof of \cref{thm1}. The following lemma restricts our attention exclusively to autocorrelations with odd shifts.
\begin{lem}\label{evenk}
For all $m \in \mathbb{N}$ and even $k \in \mathbb{Z}$, we have that 
\[ C_{m}(k) = P_{m}(k) = 0. \]
\end{lem}
\begin{proof}
This is Theorem 2.1 in \cite{hoholdt}.
\end{proof}
Now, we construct the aformentioned $v_{m}$. For $m \in \mathbb{N}$ and fixed $k_{m}$ with $0 \leq k_{m} \leq 2^{m}$ and $k_{m}$ odd, define
\begin{equation}\label{k}
\begin{aligned}
    k_{m}'  &= 2^{m}-k_{m}, \\
    k_{m-1} &= \begin{cases} k_{m} & \text{if $k_{m} \leq 2^{m-1}$} \\ k_{m}' & \text{else} \end{cases}.
\end{aligned}
\end{equation}
We wish to split $[0,2^{m}]$ into four equal-length subintervals. For $1 \leq n \leq 4$, define the open intervals
\begin{equation}\label{sdef}
    S_{m}^{n} = ((n-1)2^{m-2},n2^{m-2}).
\end{equation}
We turn our attention to
\[ v_{m}(k_{m}) := v_{m} := \begin{bmatrix} C_{m}(k_{m}) \\ C_{m}(k_{m}') \\ C_{m-1}(k_{m-1}) \end{bmatrix}. \]
\\
\begin{lem}\label{lma2} For $m \geq 3$
and 
\[ M = \begin{bmatrix} 0 & 1 & 2 \\ 0 & -1 & 2 \\ 1 & 0 & 0 \end{bmatrix}, \hspace{0.25cm} A = \begin{bmatrix} 0 & 1 & 0 \\ 1 & 0 & 0 \\ 0 & 0 & 1 \end{bmatrix}, \hspace{0.25cm} B = \begin{bmatrix} 1 & 0 & 0 \\ 0 & 1 & 0 \\ 0 & 0 & 0 \end{bmatrix}, \]
there exist $a_{i},b_{i},c \in \lbrace 0,1 \rbrace$ for $i = 3,\dots,m$ such that
\begin{align*}
    v_{m} = \left( \prod_{i=3}^{m} A^{a_{i}}MB^{b_{i}} \right) A^{c} \begin{bmatrix} 1 \\ -1 \\ 1 \end{bmatrix}.
\end{align*}
\end{lem}
\begin{proof} Let $m \geq 3$ and $0 \leq k_{m} \leq 2^{m}$. H{\o}holdt, Jensen, and Justesen in Theorem 2.2 of \textbf{\cite{hoholdt}} showed, for $m \geq 3$, that
\begin{align}
    &C_{m}(k_{m}) = C_{m-1}(2^{m-1}-k_{m}) \hspace{0.15cm} &\text{if}\hspace{0.25cm} &k_{m} \in S_{m}^{1} \label{one} \\
    &C_{m}(k_{m}) = C_{m-1}(2^{m-1}-k_{m}) + 2C_{m-2}(2^{m-1}-k_{m}) \hspace{0.15cm} &\text{if} \hspace{0.25cm} &k_{m} \in S_{m}^{2} \label{two} \\
    &C_{m}(k_{m}) = -C_{m-1}(k_{m}-2^{m-1}) + 2C_{m-2}(k_{m}-2^{m-1}) \hspace{0.15cm} &\text{if} \hspace{0.25cm} &k_{m} \in S_{m}^{3} \label{three} \\
    &C_{m}(k_{m}) = -C_{m-1}(k_{m}-2^{m-1}) \hspace{0.15cm} &\text{if} \hspace{0.25cm} &k_{m} \in S_{m}^{4}. \label{four}
\end{align}
Let $k_{m} \in S_{m}^{1}$ as defined in \eqref{sdef}. We see that $k_{m-1} = k_{m}$, which implies $k_{m-1}' = 2^{m-1}-k_{m-1} = 2^{m-1}-k_{m}$. Using this along with the relations above, we get
\begin{align*}
    C_{m}(k_{m}) &= C_{m-1}(k_{m-1}'), && \text{using \eqref{one}}\\
    C_{m}(k_{m}') &= -C_{m-1}(k_{m-1}'). && \text{using \eqref{four}}
\end{align*}
Thus,
\[ v_{m} = \begin{bmatrix} 0 & 1 & 0 \\ 0 & -1 & 0 \\ 1 & 0 & 0 \end{bmatrix} \begin{bmatrix} C_{m-1}(k_{m-1}) \\ C_{m-1}(k_{m-1}') \\ C_{m-2}(k_{m-2}) \end{bmatrix} = \begin{bmatrix} 0 & 1 & 0 \\ 0 & -1 & 0 \\ 1 & 0 & 0 \end{bmatrix}v_{m-1}. \]
\\
Now, let $k_{m} \in S_{m}^{2}$. We see that $k_{m-1} = k_{m}$, so using \eqref{two} and \eqref{three} respectively, we get
\begin{align*}
    C_{m}(k_{m}) &= C_{m-1}(k_{m-1}')+2C_{m-2}(k_{m-1}') = C_{m-1}(k_{m-1}')+2C_{m-2}(k_{m-2}),\\
    C_{m}(k_{m}') &= -C_{m-1}(k_{m-1}')+2C_{m-2}(k_{m-1}') = -C_{m-1}(k_{m-1}')+2C_{m-2}(k_{m-2}).
\end{align*}
Thus,
\[ v_{m} = \begin{bmatrix} 0 & 1 & 2 \\ 0 & -1 & 2 \\ 1 & 0 & 0 \end{bmatrix} \begin{bmatrix} C_{m-1}(k_{m-1}) \\ C_{m-1}(k_{m-1}') \\ C_{m-2}(k_{m-2}) \end{bmatrix} = \begin{bmatrix} 0 & 1 & 2 \\ 0 & -1 & 2 \\ 1 & 0 & 0 \end{bmatrix}v_{m-1}. \]
\\
Now, let $k_{m} \in S_{m}^{3}$. We see that $k_{m-1} = k_{m}'$, so using \eqref{three} and \eqref{two} respectively, we get
\begin{align*}
    C_{m}(k_{m}) &= -C_{m-1}(k_{m-1}')+2C_{m-2}(k_{m-1}') = -C_{m-1}(k_{m-1}')+2C_{m-2}(k_{m-2}), \\
    C_{m}(k_{m}') &= C_{m-1}(k_{m-1}') + 2C_{m-2}(k_{m-1}') = C_{m-1}(k_{m-1}') + 2C_{m-2}(k_{m-2}).
\end{align*}
Thus,
\[ v_{m} = \begin{bmatrix} 0 & -1 & 2 \\ 0 & 1 & 2 \\ 1 & 0 & 0 \end{bmatrix} \begin{bmatrix} C_{m-1}(k_{m-1}) \\ C_{m-1}(k_{m-1}') \\ C_{m-2}(k_{m-2}) \end{bmatrix} = \begin{bmatrix} 0 & -1 & 2 \\ 0 & 1 & 2 \\ 1 & 0 & 0 \end{bmatrix}v_{m-1}. \]
\\
Now, let $k_{m} \in S_{m}^{4}$. We see that $k_{m-1} = k_{m}'$, so using \eqref{four} and \eqref{one} respectively, we get
\begin{align*}
    C_{m}(k_{m}) &= -C_{m-1}(k_{m-1}'),\\
    C_{m}(k_{m}') &= C_{m-1}(k_{m-1}').
\end{align*}
Thus,
\[ v_{m} = \begin{bmatrix} 0 & -1 & 0 \\ 0 & 1 & 0 \\ 1 & 0 & 0 \end{bmatrix} \begin{bmatrix} C_{m-1}(k_{m-1}) \\ C_{m-1}(k_{m-1}') \\ C_{m-2}(k_{m-2}) \end{bmatrix} = \begin{bmatrix} 0 & -1 & 0 \\ 0 & 1 & 0 \\ 1 & 0 & 0 \end{bmatrix}v_{m-1}. \]
\\
In summary, we have shown that
\begin{equation}\label{Tresult}
    v_{m} = Tv_{m-1}
\end{equation}
where
\begin{equation}\label{T}
T = \begin{cases} MB & k_{m} \in S_{m}^{1} \\ M & k_{m} \in S_{m}^{2} \\ AM & k_{m} \in S_{m}^{3} \\ AMB & k_{m} \in S_{m}^{4} \end{cases}.
\end{equation}
Finally, let
\[ v = \begin{bmatrix}
1 \\ -1 \\ 1
\end{bmatrix}. \]
We see that
\begin{equation}\label{init}
    v_{2} = \begin{cases} v & k_{3} \in S_{3}^{1} \cup S_{3}^{4} \\ Av & k_{3} \in S_{3}^{2} \cup S_{3}^{3}  \end{cases}.
\end{equation}
Inductively applying \eqref{Tresult}, we are done.
\end{proof}
\vspace{0.25cm}
\indent We may express $v_{m}$ more simply. We will express $v_{m}$ as a matrix-vector product depending only on the matrices $MA$ and $MB$.
\begin{lem}\label{lma3} For $m \geq 3$ and
\[ M = \begin{bmatrix} 0 & 1 & 2 \\ 0 & -1 & 2 \\ 1 & 0 & 0 \end{bmatrix}, \hspace{0.25cm} A = \begin{bmatrix} 0 & 1 & 0 \\ 1 & 0 & 0 \\ 0 & 0 & 1 \end{bmatrix}, \hspace{0.25cm} B = \begin{bmatrix} 1 & 0 & 0 \\ 0 & 1 & 0 \\ 0 & 0 & 0 \end{bmatrix}, \]
we have that
\[ v_{m} = A^{\delta} \left( \prod_{i=3}^{m} MA^{a_{i}}B^{b_{i}} \right) \begin{bmatrix} 1 \\ -1 \\ 1 \end{bmatrix} \]
where $\delta \in \mathbb{N}$ and $(a_{i},b_{i}) \in \lbrace (0,1),(1,0) \rbrace$. In other words, we may express $v_{m}$ as an initial vector multiplied by a product of $MA$ and $MB$.
\end{lem}
\begin{proof} By \cref{lma2}, we have that
\[ v_{m} =  (T_{m}T_{m-1}\cdots T_{3}) A^{c} \begin{bmatrix} 1 \\ -1 \\ 1 \end{bmatrix} \]
where
\[ T_{i} \in \lbrace MB, M , AM, AMB \rbrace  \]
for all $i$. In the notation of \eqref{k} and \eqref{sdef}, if $k_{m} \in S_{m}^{1}$, then $k_{m-1} \in S_{m-1}^{1}$ or $k_{m-1} \in S_{m-1}^{2}$. Thus, if $T_{i} = MB$, then $T_{i-1} \in \lbrace MB, M \rbrace$ for all $3 < i \leq m$. Likewise, we consider $k_{m} \in S_{m}^{n}$ for $2 \leq n \leq 4$ and observe that
\begin{equation}\label{Tpossibilities}
    T_{i}T_{i-1} \in \lbrace MBMB, MBM, MAM, MAMB, AMAM, AMAMB, AMBMB, AMBM \rbrace 
\end{equation}
for all $3 < i \leq m$. We will now prove the statement of the lemma. We proceed by induction on $m$. The case $m=3$ follows from \eqref{T} and \eqref{init}. Assume the appropriate inductive hypothesis. Then, for $m \geq 4$, we use \Cref{lma2} and our inductive hypothesis to get
\begin{align*}
    v_{m} = A^{a}MB^{b} A^{\delta} \left( \prod_{i=3}^{m-1} MA^{a_{i}}B^{b_{i}} \right) \begin{bmatrix} 1 \\ -1 \\ 1 \end{bmatrix}
\end{align*}
with $a,b \in \lbrace 0, 1 \rbrace$. By \eqref{Tpossibilities}, we may only have
\[ A^{a}MB^{b}A^{\delta} \in \lbrace MB, MA, AMB, AMA \rbrace, \]
so $(b,\delta_{1}) \in \lbrace (0,1),(1,0) \rbrace$. Thus, the leftmost factor of our product is of the form $A^{a}MB$ or $A^{a}MA$ and we are done.
\end{proof}
\begin{rmk} Let
\[ S = \begin{bmatrix} 0 & 0 & 1 \\ 0 & 1 & 0 \\ 1 & 0 & 0 \end{bmatrix}. \]
The matrix $M_{1}$ in Lemma 3 of \cite{allouche} is our $SMAS$, and the matrix $B$ is our $SMBS$. Since $S$ is an isometry and $S^{2} = I$, the matrix products in \cite{allouche} and \cite{choi} are exactly the same as ours in norm. Due to this matrix similarity, we can pass to the calculations given in \cite{choi}. However, we will take a different approach to the computations through \cref{lma4} and \cref{lma5}. The nature of the proof of \cref{lma4} is asymptotic and will require a check of several base cases. We note that the number of base cases is small enough as to not require a computer. Throughout the rest of this paper, we denote $\| \cdot \|_{2}$ by $\| \cdot \|$.
\end{rmk}
\[ \]
\begin{lem}\label{lma4} For $j_{n}, k_{n} \in \mathbb{N}$, we have
\[ \left\| \prod_{n=1}^{2} (MA)^{j_{n}}(MB)^{k_{n}} \right\| \leq  \prod_{n=1}^{2} \lambda^{j_{n}+k_{n}}. \]
\end{lem}
\begin{proof} The $k_{n}$ we use for this lemma are not related to the $k_{m}$ described in \eqref{k}. Note that for $k_{n} \geq 2$, we have $(MB)^{k_{n}} = \pm (MB)^{2}$, so we may assume $1 \leq k_{n} \leq 2$ without loss of generality. We find that $(MA)^{j_{n}} = PDP^{-1}$ with
\begin{equation}\label{diag}
\begin{aligned}
    P &= \begin{bmatrix} 2-\lambda^{2} & 2-\overline{\nu}^{2} & 2-\nu^{2} \\ -\lambda & -\overline{\nu} & -\nu \\ 1 & 1 & 1 \end{bmatrix}, \hspace{1cm} \\
D &= \left( (-1)^{j_{n}} \begin{bmatrix} \lambda^{j_{n}} & 0 & 0 \\ 0 & \overline{\nu}^{j_{n}} & 0 \\ 0 & 0 & \nu^{j_{n}} \end{bmatrix} \right),
\\
P^{-1} &= \left( \frac{1}{\gamma} \begin{bmatrix}  (\nu - \overline{\nu}) & (\overline{\nu} - \nu)(\overline{\nu} + \nu) & (\overline{\nu} - \nu)(2+\overline{\nu} \nu) \\
     (\lambda - \nu) & (\nu - \lambda)(\nu + \lambda) & (\nu - \lambda)(2+\nu \lambda) \\
     (\overline{\nu} - \lambda) & (\lambda - \overline{\nu})(\lambda + \overline{\nu}) & (\lambda - \overline{\nu})(2+\lambda \overline{\nu}) \end{bmatrix} \right),
\end{aligned}
\end{equation}
where $\lambda = 1.659\cdots$ and $\nu = -1.329 \cdots - 0.802 \cdots i$ are roots of $x^{3}+x^{2}-2x-4$, and 
\[ \gamma = (\lambda - \overline{\nu})(\lambda - \nu)(\overline{\nu} - \nu) =  \sqrt{-236}. \]
Let 
\begin{align}\label{lambdadef}
    \lambda_{3i} = \lambda, \hspace{0.25cm} \lambda_{3i+1} = \overline{\nu}, \hspace{0.25cm} \lambda_{3i+2} = \nu
\end{align}
for all $i \in \mathbb{Z}$. Let $T_{ij}$ denote the $ij$-th entry of a matrix $T$ where $i,j \geq 1$. Let $\gamma P^{-1} = \left[ p_{ij} \right]$. We see that
\begin{equation}\label{p}
\begin{aligned}
    p_{i1} &:= \gamma(P^{-1})_{i1} = \lambda_{i+1} - \lambda_{i}, \\
    p_{i2} &:= \gamma(P^{-1})_{i2} = (\lambda_{i} - \lambda_{i+1})(\lambda_{i} + \lambda_{i+1}),\\
    p_{i3} &:= \gamma(P^{-1})_{i3} = (\lambda_{i} - \lambda_{i+1})(2+\lambda_{i}\lambda_{i+1})
\end{aligned}
\end{equation}
for $1 \leq i \leq 3$. Note also that 
\begin{equation}\label{MB^k}
    (MB)^{k_{n}} = \begin{bmatrix} 0 & -(-1)^{k_{n}} & 0 \\ 0 & (-1)^{k_{n}} & 0 \\ 2-k_{n} & k_{n}-1 & 0 \end{bmatrix}
\end{equation}
as $1 \leq k_{n} \leq 2$. For the rest of the proof, we use $j = j_{n}$ and $k = k_{n}$ for legibility. We use \eqref{diag}, \eqref{p}, and \eqref{MB^k} to see that
\begin{equation}\label{MA^jMB^k}
    \begin{aligned}
        (MA)^{j}(MB)^{k} = \frac{(-1)^{j}}{\gamma}\begin{bmatrix*}[l] (2-k) \sum (2\lambda_{i-1}^{j}-\lambda_{i-1}^{j+2})p_{i3} & \sum (2\lambda_{i-1}^{j}-\lambda_{i-1}^{j+2})((-1)^{k}(p_{i2}-p_{i1})+(k-1)p_{i3}) & 0 \\ (2-k) \sum -\lambda_{i-1}^{j+1}p_{i3} & \sum -\lambda_{i-1}^{j+1}((-1)^{k}(p_{i2}-p_{i1})+(k-1)p_{i3}) & 0 \\ (2-k) \sum \lambda_{i-1}^{j}p_{i3} & \sum \lambda_{i-1}^{j}((-1)^{k}(p_{i2}-p_{i1})+(k-1)p_{i3}) & 0  \end{bmatrix*},
    \end{aligned}
\end{equation}
where all summations are taken over $1 \leq i \leq 3$. We now consider a matrix $W(j,k) := W \in \mathbb{Z}^{3 \times 3}$ defined by
\[ (MA)^{j}(MB)^{k} = (-\lambda)^{j}W. \]
We will show that $\| (1/\lambda)W \| \leq 1$ for $j \geq 22$ so that
\[ \| (MA)^{j}(MB)^{k} \| = \| (-\lambda)^{j}W \| \leq \lambda^{j+1} \leq \lambda^{j+k}   \]
for $j \geq 22$, and afterwards we will take care of the other cases. We will rely on the Frobenius norm $\| \cdot \|_{F}$ and make use of the well-known fact that
\[ \| W \| \leq \| W \|_{F} = \left( \sum_{i,j} |W_{ij}|^{2} \right)^{1/2}. \]
We use \eqref{lambdadef} and \eqref{MA^jMB^k} to get that
\begin{align*}
     |\gamma W_{32}| &= \left| \sum_{i=1}^{3} \left( \frac{\lambda_{i-1}}{\lambda} \right)^{j}((-1)^{k}(p_{i2}-p_{i1})+(k-1)p_{i3}) \right| \\
    &= \left| 2\text{Im}(\overline{\nu})((-1)^{k}(2\text{Re}(\nu)+1) + (k-1)(2+|\nu|^{2})) + 2\text{Im}\left( \left(\frac{\overline{\nu}}{\lambda}\right)^{j}(\nu - \lambda)((-1)^{k}(\nu + \lambda + 1) + (k-1)(2+\nu \lambda)) \right) \right| \\
    &\leq \left| 2\text{Im}(\nu)(2\text{Re}(\nu)+|\nu|^{2}+3) \right| + \left| 2\left(\frac{\overline{\nu}}{\lambda}\right)^{j}(\nu - \lambda)((-1)^{k}(\nu + \lambda + 1) + (k-1)(2+\nu \lambda)) \right|.
\end{align*}
We go through the same calculations for the rest of the entries of $W$ and find that
\begin{align*}
    |\gamma W_{11}| &\leq \left|2(2-\lambda^{2})\text{Im}(\nu)(2+|\nu|^{2})\right| + \left| 2\left(\frac{\overline{\nu}}{\lambda}\right)^{j}(2-\overline{\nu}^{2})(\nu-\lambda)(2+\nu \lambda) \right|, \\
    |\gamma W_{21}| &\leq \left|2\lambda\text{Im}(\nu)(2+|\nu|^{2})\right|+\left| 2\left(\frac{\overline{\nu}}{\lambda}\right)^{j}\overline{\nu}(\nu-\lambda)(2+\nu \lambda)  \right|, \\
    |\gamma W_{31}| &\leq \left|2\text{Im}(\nu)(2+|\nu|^{2})\right|+ \left| 2\left(\frac{\overline{\nu}}{\lambda}\right)^{j}(\nu-\lambda)(2+\nu \lambda)  \right|, \\
    |\gamma W_{12}| &\leq \left|2(2-\lambda^{2})\text{Im}(\nu)(2\text{Re}(\nu)+|\nu|^{2}+3)\right| + \left| 2\left(\frac{\overline{\nu}}{\lambda}\right)^{j} (2-\overline{\nu}^{2})(\nu - \lambda)((-1)^{k}(\nu + \lambda + 1) + (k-1)(2+\nu \lambda)) \right|, \\
    |\gamma W_{22}| &\leq \left|2\lambda\text{Im}(\nu)(2\text{Re}(\nu)+|\nu|^{2}+3)\right| + \left| 2 \left(\frac{\overline{\nu}}{\lambda}\right)^{j} \overline{\nu} (\nu - \lambda)((-1)^{k}(\nu + \lambda + 1) + (k-1)(2+\nu \lambda)) \right|, \\
    |\gamma W_{32}| &\leq \left|2\text{Im}(\nu)(2\text{Re}(\nu)+|\nu|^{2}+3)\right| + \left| 2\left(\frac{\overline{\nu}}{\lambda}\right)^{j}(\nu - \lambda)((-1)^{k}(\nu + \lambda + 1) + (k-1)(2+\nu \lambda)) \right| .
\end{align*}
\\
We will focus again on bounding $|\gamma W_{32}|$. Note that when $k=2$, we have
\[ \left| (\nu - \lambda)(\nu + \lambda + 1 + (2+\nu \lambda)) \right| = 7.460\cdots \]
and when $k=1$, we have
\[ \left| (\nu - \lambda)(\nu + \lambda + 1) \right| = 4.804\cdots \]
With this, we use rational approximations to get
\[ \left| 2\left(\frac{\overline{\nu}}{\lambda}\right)^{j}(\nu - \lambda)(\nu + \lambda - 1 + (k-1)(\nu \lambda)) \right| \leq 2\left| \frac{\nu}{\lambda} \right|^{j} \cdot (7.461) \leq (0.936)^{j}(14.922). \]
Thus, we achieve the bound
\[ |\gamma W_{32}| \leq \left| 2\text{Im}(\nu)(2\text{Re}(\nu)+|\nu|^{2}+3) \right| + (0.936)^{j}(14.922) \leq 4.416 + (0.936)^{j}(14.922). \]
Similarly, we get bounds for the rest of the $|\gamma W_{ij}|$ above and divide by $|\gamma|$ and $\lambda$ to get
\begin{align*}
    |W_{11}/\lambda| &\leq 0.210 + (0.936)^{j}(0.755)\\
    |W_{21}/\lambda| &\leq 0.462 + (0.936)^{j}(0.509)\\
    |W_{31}/\lambda| &\leq 0.278 + (0.936)^{j}(0.328)\\
    |W_{12}/\lambda| &\leq 0.131 + (0.936)^{j}(1.352)\\
    |W_{22}/\lambda| &\leq 0.288 + (0.936)^{j}(0.910) \\
    |W_{32}/\lambda| &\leq 0.174 + (0.936)^{j}(0.586).
\end{align*}
We find that $\| (1/\lambda)W \|_{F} = \left( \sum_{i,j} (W_{ij}/\lambda)^{2} \right)^{1/2} \leq 0.970 < 1$ if $j \geq 22$. This gives us that
\[ \| (MA)^{j}(MB)^{k} \| = \| (-\lambda)^{j}W \| \leq \| (-\lambda)^{j}W \|_{F} \leq \lambda^{j+1} \leq 0.970 \cdot \lambda^{j+k} \]
for $j \geq 22$ and $1 \leq k \leq 2$. A computation using \texttt{Sage} improves this and shows that $\| (MA)^{j}(MB)^{k} \| \leq 0.970 \cdot \lambda^{j+k}$ holds for all $2 \leq j \leq 21$ and $1 \leq k \leq 2$ and for $(j,k) = (1,2)$. The only exception is when $(j,k) = (1,1)$, where
\[ \lambda^{2} < \left\| MAMB \right\| \leq 1.028\lambda^{2}. \]
Since $0.970 < \frac{1}{1.028}$, we use the submultiplicativity of $\| \cdot \|$ to get that\
\begin{align}\label{five}
    \left\| \prod_{n=1}^{2} (MA)^{j_{n}}(MB)^{k_{n}} \right\| \leq \prod_{n=1}^{2} \lambda^{j_{n}+k_{n}}
\end{align}
when $j_{n}, k_{n} \geq 1$ and when only one of $(j_{1},k_{1})$ or $(j_{2},k_{2})$ is equal to $(1,1)$. It is readily seen that
\[ \left\| (MAMB)^{2} \right\| \leq \lambda^{4} \]
and so we conclude that \eqref{five} holds for all $j_{n}, k_{n} \geq 1$.
\end{proof}
\begin{rmk}\label{rmk2} Very similar to our diagonalization of $MA$ are the following diagonalizations of $M$ and $AM$:
\begin{equation}\label{morediag}
\begin{aligned}
    M &= \begin{bmatrix} \lambda & \nu & \overline{\nu} \\ \lambda^{2} - 2 & \nu^{2} - 2 & \overline{\nu}^{2} - 2 \\ 1 & 1 & 1 \end{bmatrix}
    \begin{bmatrix} \lambda & 0 & 0 \\ 0 & \nu & 0 \\ 0 & 0 & \overline{\nu} \end{bmatrix}
    \left( \frac{1}{\gamma} \begin{bmatrix} (\overline{\nu}-\nu)(\nu + \overline{\nu}) & (\nu-\overline{\nu})  & (\nu-\overline{\nu})(2+\nu \overline{\nu}) \\ 
    (\lambda - \overline{\nu})(\lambda + \overline{\nu}) & (\overline{\nu}-\lambda) & (\overline{\nu}-\lambda)(2+\lambda \overline{\nu}) \\
    (\nu - \lambda)(\lambda + \nu) & (\lambda - \nu) & (\lambda - \nu)(2+\lambda \nu) \end{bmatrix} \right),
    \\
    AM &= \begin{bmatrix} -\lambda & -\overline{\nu} & -\nu \\ 2-\lambda^{2} & 2-\overline{\nu}^{2} & 2-\nu^{2} \\ 1 & 1 & 1 \end{bmatrix} 
    \begin{bmatrix} -\lambda & 0 & 0 \\ 0 & -\overline{\nu} & 0 \\ 0 & 0 & -\nu \end{bmatrix}
    \left( \frac{1}{\gamma} \begin{bmatrix} (\nu - \overline{\nu})(\nu + \overline{\nu}) & (\overline{\nu}-\nu)  & (\nu-\overline{\nu})(2+\nu \overline{\nu}) \\ 
    (\lambda - \overline{\nu})(\lambda + \overline{\nu}) & (\overline{\nu}-\lambda) & (\lambda - \overline{\nu})(2+\lambda \overline{\nu}) \\
    (\nu - \lambda)(\lambda + \nu) & (\lambda - \nu) & (\nu - \lambda)(2+\lambda \nu) \end{bmatrix} \right).
\end{aligned}
\end{equation}
\end{rmk}
\begin{lem}\label{lma5} For $j,k \in \mathbb{N} \cup \lbrace 0 \rbrace$, we have
\[ \left\| (MA)^{j}(MB)^{k} \right\| \ll  \lambda^{j+k}, \]
where the implicit constant does not depend on $j$ or $k$.
\end{lem}
\begin{proof} Again, note that for $k \geq 2$, we have $(MB)^{k} = \pm (MB)^{2}$, so we may assume $k \leq 2$ without loss of generality. Using the diagonalization found for $MA$ in \cref{lma4}, we have that $\| (MA)^{j} \| \ll \lambda^{j}$ and this result follows immediately.
\end{proof}
\noindent The following lemma is used solely for proving the lower bound in \cref{thm1}.
\\
\begin{lem}\label{lma6} Fix $m \in \mathbb{N}$. If $m$ is odd, then
\[ \left\lfloor \frac{2^{m+1}}{3} \right\rfloor = 2^{m+1} - \left\lceil \frac{2^{m+2}}{3} \right\rceil. \]
Similarly, if $m$ is even, then
\[ \left\lceil \frac{2^{m+1}}{3} \right\rceil = 2^{m+1} - \left\lfloor \frac{2^{m+2}}{3} \right\rfloor. \]
\end{lem}
\begin{proof}
Suppose $m \in \mathbb{N}$ is odd. For $x \in \mathbb{R}$, let $\lbrace x \rbrace$ denote the fractional part of $x$. Then,
\begin{align*}
    \left\lfloor \frac{2^{m+1}}{3} \right\rfloor + \left\lceil \frac{2^{m+2}}{3} \right\rceil &= \frac{2^{m+1}}{3} - \left\lbrace \frac{2^{m+1}}{3} \right\rbrace + \frac{2^{m+2}}{3} + 1 - \left\lbrace \frac{2^{m+2}}{3} \right\rbrace \\
    &= \frac{2^{m+1}+2^{m+2}}{3} \\
    &= 2^{m+1}
\end{align*}
and so we are done in this case. The proof of the lemma follows similarly for even $m$.
\end{proof}
\vspace{0.25cm}
\begin{rmk}\label{rmk3} For $x \in \mathbb{R}$ with fractional part $\lbrace x \rbrace \neq 1/2$, denote by $\lfloor x \rceil$ the nearest integer to $x$. In the notation of \eqref{k} and \eqref{sdef}, if we pick $k_{m} = \left\lfloor \frac{2^{m+1}}{3} \right\rceil$, then we have that $k_{m} \in  (3\cdot2^{m-2},2^{m}) = S^{m}_{3}$. \cref{lma6} tells us that 
\begin{align*}
    k_{m-1} = k_{m}' &= \left\lfloor \frac{2^{m}}{3} \right\rceil \in S_{3}^{m-1} \\
    k_{m-2} = k_{m-1}' &= \left\lfloor \frac{2^{m-1}}{3} \right\rceil \in S_{3}^{m-2} \\
    &\vdots \\
    k_{3} = k_{4}' &= 5 \in S_{3}^{3}
\end{align*}
so that by \cref{lma2} we have
\[ v_{m} = \begin{bmatrix}
C_{m}(k_{m}) \\
C_{m}(k_{m}') \\
C_{m-1}(k_{m-1})
\end{bmatrix} = (AM)^{m-3} \begin{bmatrix}
-1 \\ 1 \\ -1
\end{bmatrix}. \]
\end{rmk}
We are now ready to prove the main theorem, which we recite below.
\\
\\
\textbf{Theorem 1.} \textit{For all $m \in \mathbb{N}$, there exist $K_{1},K_{2} > 0$ such that
\[ K_{1}\lambda^{m} \hspace{0.1cm} \leq \hspace{0.1cm} \max_{k \neq 0} \hspace{0.1cm} |C_{m}(k)| \hspace{0.1cm} \leq \hspace{0.1cm} K_{2}\lambda^{m}, \]
where $\lambda = 1.659\cdots$ is the real root of $x^{3}+x^{2}-2x-4$.}
\begin{proof} A quick computation shows that $\max_{k \neq 0} |C_{m}(k)| \neq 0$ for $1 \leq m \leq 2$. Fix $m \geq 3$. First, we focus on the upper bound. For even $k$, we have that $C(m,k) = 0$ by \cref{evenk}. Using the notation in \eqref{k}, Fix $k_{m}$ so that $0 \leq k_{m} \leq 2^{m}$ and $k_{m}$ is odd, and let
\[ v_{m} = \begin{bmatrix} C_{m}(k_{m}) \\ C_{m}(k_{m}') \\ C_{m-1}(k_{m-1}) \end{bmatrix}. \]
The idea is to use the fact that
\[ |C_{m}(k_{m})| \leq \|v_{m}\|.\]
By \cref{lma3}, we have that
\[ v_{m}  = A^{\delta} \left( \prod_{i=3}^{m} MA^{a_{i}}B^{b_{i}} \right) \begin{bmatrix} 1 \\ -1 \\ 1 \end{bmatrix} \]
where $\delta \in \mathbb{N}$ and $(a_{i},b_{i}) \in \lbrace (0,1), (1,0) \rbrace$. We see that
\[ \left\| A^{\delta} \left( \prod_{i=3}^{m} MA^{a_{i}}B^{b_{i}} \right) \begin{bmatrix} 1 \\ -1 \\ 1 \end{bmatrix} \right\| \ll  \left\| A^{\delta} \left( \prod_{i=3}^{m} MA^{a_{i}}B^{b_{i}} \right) \right\| = \left\| \prod_{i=3}^{m} MA^{a_{i}}B^{b_{i}} \right\|. \]
Note that $\| MBv \| \leq \| MAv \|$ for all $v \in \mathbb{R}^{3}$, so we assume that $(a_{m},b_{m}) = (1,0)$ without loss of generality. With this assumption, we have that
\[ \prod_{i=3}^{m} MA^{a_{i}}B^{b_{i}} = \left( \prod_{i=1}^{n} (MA)^{\alpha_{i}}(MB)^{\beta_{i}} \right) (MA)^{\ell} \]
where $\alpha_{i},\beta_{i} \in \mathbb{N} \cup \lbrace 0 \rbrace$ and $\ell + \sum_{i=1}^{n} \alpha_{i} + \beta_{i} = m-3$. We use \cref{lma4} and \cref{lma5} to conclude that
\[ |C_{m}(k_{m})| \leq \| v_{m} \| \ll \left\| \prod_{i=3}^{m} MA^{a_{i}}B^{b_{i}} \right\| \ll \left\| \prod_{i=1}^{n} (MA)^{\alpha_{i}}(MB)^{\beta_{i}} \right\|\cdot \lambda^{\ell} \ll \lambda^{\ell}\prod_{i=1}^{n} \lambda^{\alpha_{i} + \beta_{i}} \ll \lambda^{m} \]
where the implicit constants are independent of $m,n$.
\\
\\
Now, we concentrate on the lower bound. For this, we use the same idea as the authors in \cite{allouche}; namely, we exhibit $C_{m}(\ell_{m})$ for a specific $\ell_{m}$ such that $|C_{m}(\ell_{m})| \gg \lambda^{m}$. Let 
\[ \ell_{m} = \left\lfloor \frac{2^{m+1}}{3} \right\rceil \]
where $\left\lfloor x \right\rceil$ denotes the nearest integer to $x$. By \cref{lma2} and \cref{lma6} (see \cref{rmk3}), we have that
\[ C_{m}(\ell_{m}) = \begin{bmatrix} 1 & 0 & 0 \end{bmatrix} (AM)^{m-2} \begin{bmatrix} -1 \\ 1 \\ 1 \end{bmatrix}. \]
Using this equation and the diagonalization of $AM$ given in \eqref{morediag}, we find that there exist some $a,b,c \in \mathbb{C}$ so that
\begin{align*}
    C(m,\ell_{m}) = a(-\lambda)^{m-1} + b(-\overline{\nu})^{m-1} + c(-\nu)^{m-1}
\end{align*}
where $\nu = -1.329\cdots-0.802\cdots i$ is a root of $x^{3}+x^{2}-2x-4$. Using the diagonalization of $AM$ given in $\eqref{morediag}$, we find that
\[ a = \frac{-2\text{Re}(\nu)(2\text{Re}(\nu)+|\nu|^{2}-1)}{\gamma} \neq 0. \]
We also have that $C_{m}(\ell_{m}) \neq 0$ as $\ell_{m}$ is odd. In other words,
\begin{equation*} a(-\lambda)^{m-1}+b(-\overline{\nu})^{m-1} + c(-\nu)^{m-1} \neq 0,
\end{equation*}
so
\begin{equation}\label{lower}
C(m,\ell_{m}) = |a(-\lambda)^{m-2} + b(-\overline{\nu})^{m-2} + c(-\nu)^{m-2}| \gg \lambda^{m}.
\end{equation}
This concludes the proof.
\end{proof}
\begin{rmk} In  Theorem 2.2 of \cite{hoholdt}, the authors provide relations \eqref{one}-\eqref{four} used in our \cref{lma2} for a class of sequences that is more general than the class of Rudin-Shapiro sequences. In particular, they consider the class of sequences $(a_{0}, a_{1}, \dots, a_{2^{m}-1})$ with
\begin{align*}
    a_{0} &= 1, \\
    a_{2^{i}+j} &= (-1)^{j+f(i)}a^{2^{i}-j-1}, && 0 \leq j \leq 2^{i}-1, \\
    & && 0 \leq i \leq m-1,
\end{align*}
for all $m \in \mathbb{N}$ and $f : \mathbb{N} \to \lbrace 0 , 1 \rbrace$ being an arbitrary function. This class of sequences is closely related to Welti codes \cite{welti, welti2}. The Rudin-Shapiro sequences are recovered by choosing $f$ so that $f(0) = f(2k-1) = 0$ and $f(2k) = 1$ for all $k \in \mathbb{N}$. They show the aforementioned relations but the right-hand side of each is multiplied by $(-1)^{f(m-1)+f(m-2)}$. Since these relations are the same as the relations in \cref{lma2} up to a factor of -1, we may obtain an analogue of \cref{lma2} with the matrices $\pm M$, $A$, and $B$. In norm, products of these matrices are the same as products of $M$, $A$, and $B$. Thus, we may obtain \cref{thm1} for autocorrelations of this more general family of sequences in exactly the same manner as we have done above for Rudin-Shapiro sequences.
\end{rmk}

As remarked in the introduction to this paper, B. Saffari mentioned in a private communication that $P_{m}(k)$ and $C_{m}(k)$ behave essentially in the same way. On this topic, we show that Theorem 1 holds analogously for $P_{m}(k)$, and we give a useful relation between the periodic and aperiodic autocorrelations.

\begin{thm}\label{thm2}
Using the notation of \eqref{sdef}, we have for $m \geq 3$ that
\begin{equation}\label{P}
    P_{m}(k) = \begin{cases} 0 & k \in S_{m}^{1} \cup S_{m}^{4}, \\ 4C_{m-2}(|2^{m-1}-k|) & k \in S_{m}^{2} \cup S_{m}^{3}. \end{cases}
\end{equation}
Additionally, for all $m \in \mathbb{N}$, we have that
\begin{equation}\label{Pmax}
    \lambda^{m} \ll \max_{k \neq 0} |P_{m}(k)| \ll \lambda^{m}
\end{equation}
where the implicit constants do not depend on $m$, and $\lambda = 1.659\cdots$ is the real root of $x^{3}+x^{2}-2x-4$.
\end{thm}
\begin{proof}
First, we prove \eqref{P}. Fix $k \in S_{m}^{1}$. By \eqref{one} and \eqref{four}, we have that 
\[ P_{m}(k) = C_{m}(k) + C_{m}(2^{m}-k) = 0.\]
By symmetry of $P_{m}(k)$, the same holds for $k \in S_{m}^{4}$. Now, fix $k \in S_{m}^{2}$. We use \eqref{two} and \eqref{three} to get that
\begin{align*}
    P_{m}(k) &= C_{m}(k) + C_{m}(2^{m}-k) \\
    &= \left[ C_{m-1}(2^{m-1}-k) + 2C_{m-2}(2^{m-1}-k) \right] + \left[ -C_{m-1}(2^{m-1} - k_{m}) + 2C_{m-2}(2^{m-1}-k) \right] \\
    &= 4C_{m-2}(2^{m-1}-k).
\end{align*}
By symmetry of $P_{m}(k)$, the same holds for $k \in S_{m}^{3}$. Thus, we have shown \eqref{P}. Now, we will prove \eqref{Pmax}. A quick computation shows that $\max_{k \neq 0} |P_{m}(k)| \neq 0$ for $1 \leq m \leq 4$. Fix $m \geq 5$. The upper bound of \eqref{Pmax} follows directly from \cref{thm1}. For the lower bound, we use \eqref{P} and \cref{lma6} to get that
\[ P_{m}\left( \left\lfloor \frac{2^{m}}{3} \right\rceil \right) = 4C_{m-2}\left( \left\lfloor \frac{2^{m-1}}{3} \right\rceil \right), \]
and by \eqref{lower}, we have 
\[ 4C_{m-2}\left( \left\lfloor \frac{2^{m-1}}{3} \right\rceil \right) \gg \lambda^{m}, \]
so we are done.
\end{proof}

\section{Connections to Joint Spectral Radius Theory}

Let $\mathcal{T}$ be a bounded set of matrices in $\mathbb{R}^{n \times n}$ and $\| \cdot \|$ be a matrix norm. For $m \geq 1$, let $\mathcal{M}^{m}$ denote the set of products of $m$ matrices in $\mathcal{T}$. We define the \emph{joint spectral radius of $\mathcal{T}$}, or $JSR(\mathcal{T})$, by
\[ \lim_{m \to \infty} \sup_{\Pi \in \mathcal{M}^{m}} \| \Pi \|^{1/m}.\]
This limit always exists and is independent of the matrix norm chosen. The case in which $\mathcal{T}$ consists of a single matrix $T$, we have Gelfand's formula:
\[ \lim_{m \to \infty} \| T^{m} \|^{1/m} = \rho(T) \]
where $\rho(T)$ denotes the spectral radius of $T$. The problem of proving \cref{thm1} is tantamount to proving that
\begin{equation}\label{jsrmamb}
    JSR(\lbrace MA, MB \rbrace) = \lambda .
\end{equation}
There are several popular algorithms for approximating and even exactly computing the joint spectral radius, some of which can be found in \cite{protasov, jungers}. A \emph{branch-and-bound} method of computing $JSR(\mathcal{T})$ consists of computing $\sup_{\Pi \in \mathcal{M}^{m}} \| \Pi \|$ for increasing $m$, eventually converging to $JSR(\mathcal{T})$. The approach taken in Section 2 of this paper can be considered a sort of branch-and-bound method of computing $JSR(\lbrace MA, MB \rbrace)$, although instead of considering finite products of matrices in $\lbrace MA, MB \rbrace$, we consider finite products of arbitrarily large powers of matrices in $\lbrace MA, MB \rbrace$. Equation \eqref{jsrmamb} tells us that 
\begin{equation}\label{niceprop}
    JSR(\lbrace MA, MB \rbrace) = \max \lbrace \rho(MA), \rho(MB) \rbrace.
\end{equation}
This property can be asserted for certain families of matrices, such as simultaneously upper triangularizable matrices, normal matrices, and symmetric  (see Section 2.3.2 in \cite{jungers}). To the best of our knowledge, $\lbrace MA, MB \rbrace$ is not in any family of matrices shown to guarantee the property \eqref{niceprop}. So, it is necessary to compute $JSR(\lbrace MA, MB \rbrace)$ by other means.
\\
\indent We will give a heuristic proof of \eqref{jsrmamb} using the ``invariant polytope algorithm" given in Section 2.1 of \cite{protasov}, a simple overview of which can be found Section 2.1 of \cite{protasov2}. First, we note that $\lbrace MA, MB \rbrace$ is an irreducible set of matrices, i.e. there is no proper subspace of $\mathbb{R}^{3}$ invariant under both $MA$ and $MB$. Irreducibility is convenient in the context of working with $JSR(\mathcal{T})$ as it implies that $JSR(\mathcal{T}) > 0$ and that $\mathcal{T}$ is nondefective, i.e. there exists $K > 0$ such that for all $m$, we have
\[ \sup \lbrace \| \Pi \| : \Pi \in \mathcal{M}^{m} \rbrace \leq K \cdot JSR(\mathcal{T})^{m}, \]
or equivalently that $\mathcal{T}$ admits an extremal norm, i.e. a vector norm $| \cdot |$ such that for all $x \in \mathbb{R}^{n}$ and $T \in \mathcal{T}$, we have 
\[ |Tx| \leq JSR(\mathcal{T})|x|. \]
On these topics, we refer the reader to \cite{jungers, wirth}. On top of these properties, irreducibility is a necessary assumption in order to use the algorithm in \cite{protasov}, which we now expound upon. We choose $MA$ as a candidate for a \emph{spectrum-maximizing product}, which generally speaking is a matrix $\Pi \in \mathcal{M}^{m}$ such that
\[ [\rho(\Pi)]^{1/m} = JSR(\mathcal{T}).\]
Using \texttt{ConvexHull} in the \texttt{SciPy} library for \texttt{Python}, we find that the invariant polytope algorithm halts in 8 steps, after which we have found a polytope $\mathcal{P}$ with 30 vertices such that
\begin{equation}\label{invar}
    (MA) \mathcal{P} \cup (MB)\mathcal{P} \subset \lambda \mathcal{P}.
\end{equation}
Of course, this is not a formal proof, and it is also possible that this output was due to numerical imprecision. The polytope $\mathcal{P}$ in $\eqref{invar}$ would be the unit ball of an extremal norm associated with $\lbrace MA, MB \rbrace$. If \eqref{invar} is true, it would tell us that $MA$ is indeed a spectrum-maximizing product, which would prove \eqref{jsrmamb}. Going through 8 steps of this algorithm requires consideration of about the same number of cases as our method, being bounded above by $2^{8} = 256$. We believe our method to be useful for computation of the $JSR$ for sets of matrices consisting mostly of matrices $T$ with $T^{n} = T$ for some $n \in \mathbb{N}$. However, as our method was created around attacking \cref{thm1}, we have no evidence of its ability to compute joint spectral radii in general, unlike the algorithm in \cite{protasov}. It is purely luck that the product of $(MA)^{j_{n}}(MB)^{k_{n}}$ in \cref{lma4} did not exceed a length of 2.

\section{Directions for Further Study}

We now concern ourselves with which $k$ gives maximal $C_{m}(k)$. Fix $m$ and suppose that $k = k_{m}^{\ast}$ gives maximal $C_{m}(k)$.
\\
\\
\noindent \textbf{Conjecture.} We have that $k_{m}^{\ast}$ is unique for each $m$ and $\displaystyle\lim_{m \to \infty} \frac{3k_{m}^{\ast}}{2^{m+1}} = 1.$
\\
\\
In other words, the $k$ that gives maximal autocorrelation of the $m$-th Rudin-Shapiro sequence is asymptotically 2/3 the length of the $m$-th Rudin-Shapiro sequence.
\\
\\
Let $\ell_{m} = \left\lfloor \frac{2^{m+1}}{3} \right\rceil$
where $\left\lfloor x \right\rceil$ denotes the nearest integer to $x$. We find that $k_{m}^{\ast}$ is unique for each $3 \leq m \leq 16$ and that:
\begin{center}
\begin{tabular}{ | m{5em} | m{5em} | m{5em} | m{11em} | } 
  \hline
  $m$ & $|k_{m}^{\ast} - \ell_{m}|$ & $k_{m}^{\ast}/\ell_{m}$ \\ 
  \hline
  3 & 2 & 0.6 \\ 
  \hline
  4 & 0 & 1 \\ 
  \hline
  5 & 8 & 0.619... \\ 
  \hline
  6 & 0 & 1 \\ 
  \hline
  7 & 34 & 0.6 \\ 
  \hline
  8 & 2 & 1.011... \\ 
  \hline
  9 & 22 & 1.064... \\ 
  \hline
  10 & 8 & 1.011... \\ 
  \hline
  11 & 0 & 1 \\ 
  \hline
  12 & 34 & 1.012... \\ 
  \hline
  13 & 86 & 1.015... \\ 
  \hline
  14 & 136 & 1.012... \\ 
  \hline
  15 & 18 & 0.999... \\ 
  \hline
  16 & 0 & 1 \\
  \hline
\end{tabular}
\end{center}
\[ \]
Note that for $m = 4,6,11,16,\dots$ we have $k_{m}^{\ast} = \ell_{m}$. That leads us to pose the following question.
\\
\\
\textbf{Question.} Assuming that $k_{m}^{\ast}$ is unique for each $m$, do there exist infinitely many integers $m$ such that $k_{m}^{\ast} = \ell_{m}$?
\\
\\
It is interesting that $|C_{m}(k)|$ generally increases as $k$ increases from $k=0$ to a local maximum at $k \approx \frac{2^{m+1}}{5}$ and $|C_{m}(k)|$ generally decreases as $k$ increases from a local maximum at $k \approx \frac{2^{m+1}}{3}$ to $k=2^{m}$ as we can see in the following graphs:
\begin{align*}
  \centering
  \begin{minipage}[b]{0.5\textwidth}
    \includegraphics[width=\textwidth]{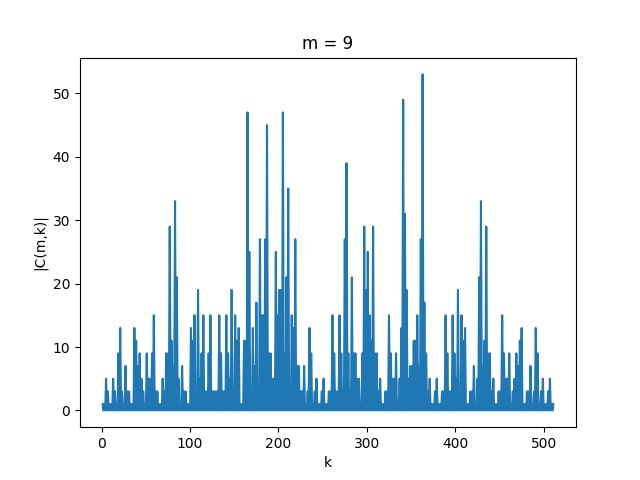}
  \end{minipage}
  \hfill
  \begin{minipage}[b]{0.5\textwidth}
    \includegraphics[width=\textwidth]{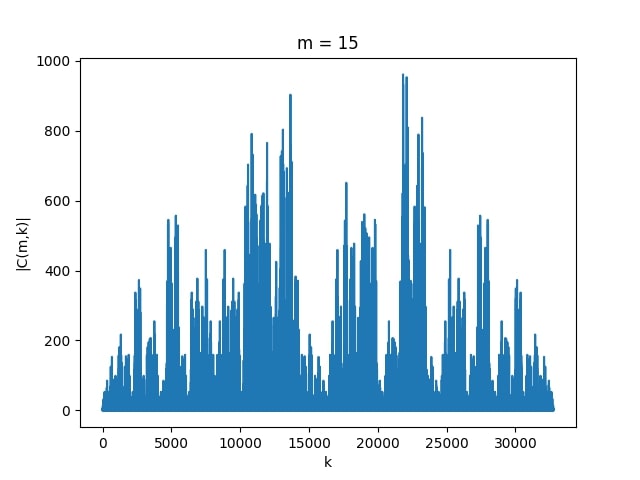}
  \end{minipage}
\end{align*}

\section*{Acknowledgements}

We would like to thank Stephen Choi for reviewing this paper and giving thorough feedback and guidance. We would also like to thank the anonymous referee and the corresponding editor for their attention to detail and substantial suggestions.

\nocite{*}
\bibliographystyle{elsarticle-num}
\bibliography{bib.bib}

\end{document}